\documentclass[a4paper,11pt]{article}

\usepackage{amsmath, amssymb, amsthm}
\usepackage{graphics}

\usepackage{psfrag, booktabs}
\usepackage{color, epsfig, float}
\usepackage[active]{srcltx}
\usepackage[dvips, bookmarksopen, colorlinks, linkcolor = blue, urlcolor = red,
            citecolor = red, menucolor = blue]{hyperref}

\newtheorem{theorem}{Theorem}[section]
\newtheorem{lemma}[theorem]{Lemma}

\newtheorem{propo}[theorem]{Proposition}
\newtheorem{remark}[theorem]{Remark}

\newcommand{\nach}{\; \longrightarrow \;}

\newcommand{\ipl}{\langle}
\newcommand{\ipr}{\rangle}

\newcommand\F{\mathrm{F}}

\newcommand\summ{\textstyle\sum\limits}

\newcommand\N{\mathbb{N}}

\DeclareMathOperator{\argmin}{arg\, min}
 \DeclareMathOperator{\nr}{\mathcal
N} 

\newcommand\norm[1]{\|#1\|}

\newcommand\set[1]{\{#1\}}

\begin{document}

\title{Modified iterated Tikhonov methods for solving systems of nonlinear ill-posed
equations} \setcounter{footnote}{1}

\author{
J.~Baumeister%
\thanks{Department of Mathematics, Goethe University,
        Robert-Mayer Str. 6-10, D-60054 Frankfurt Main, Germany
        \href{mailto:baumeist@math.uni-frankfurt.de}{\tt baumeist@math.uni-frankfurt.de}.}
\and
A.~De\,Cezaro%
\thanks{Institute of Mathematics Statistics and Physics,
        Federal University of Rio Grande, Av. Italia km 8, 96201-900
        Rio Grande, Brazil
        \href{mailto:decezaro@impa.br}{\tt decezaro@impa.br}.}
\and
A.~Leit\~ao%
\thanks{Department of Mathematics, Federal University of St. Catarina,
        P.O. Box 476, 88040-900 Florian\'opolis, Brazil
        \href{mailto:acgleitao@gmail.com}{\tt acgleitao@gmail.com}.} }
\date{\small \today}

\maketitle

\begin{abstract}
We investigate iterated Tikhonov methods coupled with a Kaczmarz strategy
for obtaining stable solutions of nonlinear systems of ill-posed operator
equations. We show that the proposed method is a convergent regularization
method.
In the case of noisy data we propose a modification, the so called
loping iterated Tikhonov-Kaczmarz method, where a sequence of relaxation
parameters is introduced and a different stopping rule is used. Convergence
analysis for this method is also provided.
\end{abstract}

\medskip


\section{Introduction} \label{sec:intro}

In this paper we propose a new method for obtaining regularized
approximations of systems of nonlinear ill-posed operator equations.

The \textit{inverse problem} we are interested in consists of
determining an unknown physical quantity $x \in X$ from the set of
data $(y_0, \dots,  y_{N-1}) \in Y^N$, where $X$, $Y$ are Hilbert
spaces and $N \geq 1$. In practical situations, we do not know the
data exactly. Instead, we have only approximate measured data
$y_i^\delta \in Y$ satisfying
\begin{equation}\label{eq:noisy-i}
    \norm{ y_i^\delta - y_i } \le \delta_i \, , \ \ i = 0, \dots, N-1 \, ,
\end{equation}
with $\delta_i > 0$ (noise level). We use the notation $\delta :=
(\delta_0, \dots, \delta_{N-1})$. The finite set of data above is obtained
by indirect measurements of the parameter, this process being described by
the model
\begin{equation}\label{eq:inv-probl}
    F_{i}(x)  =  y_{i} \, , \ \ i = 0, \dots, N-1 \, ,
\end{equation}
where $F_i: D_i \subset X \to Y$, and $D_i$ are the corresponding
domains of definition.

Standard methods for the solution of system \eqref{eq:inv-probl} are
based in the use of \textit{Iterative type} regularization methods
\cite{BakKok04, EngHanNeu96, KalNeuSch08}) or
\textit{Tikhonov type} regularization methods
\cite{EngHanNeu96, Mor93, TikArs77} after rewriting (\ref{eq:inv-probl})
as a single equation $F( x ) =  y$, where
\begin{align} \label{eq:single-op}
    F  : = ( F_0, \dots, F_{N-1} ): \bigcap\nolimits_{i=0}^{N-1} D_i \to Y^N
\end{align}
and $y := (y^0, \dots,  y^{N-1})$. However these methods become
inefficient if $N$ is large or the evaluations of $F_i(x)$ and
$F_i'(x)^\ast$ are expensive. In such a situation, Kaczmarz type
methods \cite{Kac37, Mcc77, Nat97} which cyclically
consider each equation in (\ref{eq:inv-probl}) separately are much
faster \cite{Nat97} and are often the method of choice in practice.

For recent analysis of Kaczmarz type methods for systems of
ill-posed equations, we refer the reader to
\cite{BKL09, HLR09, CHLS08, HLS07}.
The starting point of our approach is the iterated Tikhonov method
\cite{HG98, BS87, Lar75} for solving linear ill-posed problems. This
regularization method is defined by
$$
  x_{k+1}^\delta \ \in \ \argmin \big\{ \| F\, x - y^\delta \|^2
                    + \alpha \| x - x_k^\delta \|^2 \big\} \, ,
$$
what corresponds to the iteration
$$
  x_{k+1}^\delta \ = \ x_k^\delta - \alpha^{-1} F^*
                     ( F\, x_{k+1}^\delta  - y^\delta ) \, .
$$
Motivated by the ideas in \cite{BKL09, HLS07}, we propose
in this article an \textit{iterated Tikhonov-Kaczmarz method}
(\textsc{iTK} method) for solving (\ref{eq:inv-probl}). This
iterative method is defined by
\begin{equation} \label{eq:litk-tf}
x_{k+1}^\delta \ \in \ \argmin \big\{ \| F_{[k]}(x) - y_{[k]}^\delta \|^2
                      + \alpha \| x - x_k^\delta \|^2 \big\} \, .
\end{equation}
Here $\alpha > 0$ is an appropriate chosen number (see \eqref{eq:def-alp-tau}
below), $[k] := (k \mod N) \in \set{0, \dots, N-1}$, and $x_0^\delta = x_0
\in X$ is an initial guess, possibly incorporating some \textit{a
priori} knowledge about the exact solution.

\begin{remark} \label{rem:non-uniq}
Notice that from the iteration formula in \eqref{eq:litk-tf} we conclude that
\begin{equation} \label{eq:itk}
x_{k+1}^\delta \ = \ x_{k}^\delta - \alpha^{-1} F_{[k]}'(x_{k+1}^\delta)^*
                  ( F_{[k]}(x_{k+1}^\delta) - y_{[k]}^\delta ) \, .
\end{equation}
As usual for nonlinear Tikhonov type regularization, the global minimum
for the Tikhonov functionals in \eqref{eq:litk-tf} need not be unique.
For exact data we obtain the same convergence statements for any possible
sequence of iterates (see Section~\ref{sec:conv-exact}) and we will accept
any global solution. For noisy data, a (strong) semi-convergence result
is obtained under a smooth assumption on the functionals $F_i$ (see
assumption (A4) in Section~\ref{sec:conv-noisy}), which guarantees
uniqueness of global minimizers in \eqref{eq:litk-tf}.
\end{remark}

\begin{remark}
It is worth noticing that some authors consider iterated Tikhonov regularization
with the number of iterations $n \in \mathbb N$ being fixed \cite{GS00,KN08,Sch93a}.
In this case, $\alpha$ plays the role of the regularization parameter. This
regularization method is also called $n$-th iterated Tikhonov method.
\end{remark}

The \textsc{iTK} method consists in incorporating the Kaczmarz
strategy in the iterated Tikhonov method. This strategy is analog to the
one introduced in \cite{HLS07} regarding the Landweber-Kaczmarz
(\textsc{LK}) iteration, in \cite{CHLS08} regarding the
Steepest-Descent-Kaczmarz (\textsc{SDK}) iteration, in \cite{HLR09}
regarding the Expectation-Maximization-Kaczmarz (\textsc{EMK})
iteration.
As usual in Kaczmarz type algorithms, a group of $N$ subsequent steps
(starting at some multiple $k$ of $N$) shall be called a {\em cycle}.
The iteration should be terminated when, for the first time, at least
one of the residuals $\norm{ F_{[k]}(x_{k+1}^\delta) - y_{[k]}^\delta}$
drops below a specified threshold within a cycle. That is, we stop the
iteration at
\begin{equation} \label{eq:def-discrep}
  k_*^\delta \ := \ \min \{ lN \in \N: \,
  \norm{ F_i(x_{lN+i+1}^\delta) - y_i^\delta} \le \tau \delta_i \, , \
  {\rm for \ some} \ 0 \le i \le N-1 \} \, ,
\end{equation}
where $\tau>1$ still has to be chosen (see \eqref{eq:def-alp-tau} below).
Notice that for $k = k_*^\delta$ we do not necessarily have
$\norm{ F_i (x_{k_*^\delta+i}^\delta) - y_i^\delta} \le \tau \delta_i$ for
all $i = 0, \dots, N-1$.
In the case of noise free data, $\delta_i = 0$ in \eqref{eq:noisy-i},
the stop criteria in \eqref{eq:def-discrep} may never be reached, i.e.
$k_*^\delta = \infty$ for $\delta_i = 0$.


In the  case of noisy data, we also propose a loping version of
\textsc{iTK}, namely, the \textsc{l-iTK} iteration.
In the \textsc{l-iTK} iteration we omit an update of the \textsc{iTK}
iteration (within one cycle) if the corresponding $i$-th residual is
below some threshold. Consequently, the \textsc{l-iTK} method is not
stopped until all residuals are below the specified threshold. We provide
a complete convergence analysis for both \textsc{iTK} and \textsc{l-iTK}
iterations. In particular we prove that \textsc{l-iTK} is a convergent
regularization method in the sense of \cite{EngHanNeu96}.

The article is outlined as follows.
In Section \ref{sec:basic} we formulate basic assumptions and derive
some auxiliary estimates required for the analysis.
In Section \ref{sec:conv-exact} a convergence result for the \textsc{iTK}
method is proved.
In Section \ref{sec:conv-noisy} a semi-convergence result for the
\textsc{iTK} method for noisy data is proved.
In Section~\ref{sec:litk} we introduce (for the case of noisy data) a
loping version of the \textsc{iTK} method and prove a semi-convergence
result for this new method.
In Section~\ref{sec:applications} we discuss some possible applications related
to parameter identification in elliptic PDE's.
Section~\ref{sec:conclusion} is devoted to final remarks an conclusions.

\section{Assumptions and preliminary results} \label{sec:basic}

We begin this section by introducing some assumptions, that are
necessary for the convergence analysis presented in the next
section. These assumptions derive from the classical assumptions
used in the analysis of iterative regularization methods
\cite{EngHanNeu96, KalNeuSch08, Sch93a}.
\medskip

\noindent (A1) \ 
The operators $F_i$ are weakly sequentially continuous and Fr\'echet
differentiable; the corresponding domains of definition $D_i$ are
weakly closed. Moreover, we assume the existence of $x_0 \in X$, $M > 0$,
and $\rho > 0$ such that
\begin{equation} \label{eq:a-dfb}
\| F_i'(x) \| \ \le \ M \, , \quad \ x \in B_\rho(x_0)
                                    \subset \bigcap\nolimits_{i=0}^{N-1} D_i \, .
\end{equation}
Notice that $x_0^\delta =x_0$ is used as starting value of the
\textsc{iTK} iteration.
\medskip

\noindent (A2) \ 
This is an uniform assumption on the nonlinearity of the operators $F_i$.
We assume that the {\em local tangential cone condition} \cite{EngHanNeu96,
KalNeuSch08}
\begin{equation} \label{eq:a-tcc}
\| F_i(x) - F_i(\bar{x}) -  F_i'(\bar{x})( x - \bar{x}) \|_Y \ \leq \
   \eta \norm{ F_i(x)-F_i(\bar{x}) }_{Y} \, , \qquad
     x, \bar{x} \in B_{\rho}(x_0)
\end{equation}
holds for some $\eta < 1$. 
\medskip

\noindent (A3) \ 
There exists and element $x^* \in B_{\rho/4}(x_0)$ such that $F(x^*) = y$,
where $y = (y_0, \dots,  y_{N-1})$ are the exact data satisfying
\eqref{eq:noisy-i}.
\medskip

We are now in position to choose the positive constants $\alpha$
and $\tau$ in \eqref{eq:itk}, \eqref{eq:def-discrep}. For the rest of
this article we shall assume
\begin{equation} \label{eq:def-alp-tau}
\alpha \ > \ \frac{16}{3} \Big( \frac{\delta_{max}}{\rho} \Big)^2 \, , \quad\quad
\tau \ > \ \frac{1+\eta}{1-\eta} \geq 1 \, ,
\end{equation}
where $\delta_{max} := \max_j \{ \delta_j \}$.
In particular, for linear problems we can choose $\tau = 1$.
Moreover, for exact data (i.e., $\delta_j = 0$, for $j=0,\dots,N-1$)
we require simply $\alpha > 0$.
\medskip

In the sequel we verify some basic results that are necessary for
the convergence analysis derived in the next section. The first
result concerns the well-definiteness of the Tikhonov functionals
\begin{equation} \label{eq:def-Jk}
J_k(x) \ := \ \| F_{[k]}(x) - y_{[k]}^\delta \|^2
            + \alpha \| x - x_k^\delta \|^2 \, ,
\end{equation}
which obviously relate to iteration \eqref{eq:itk} due to the fact that
$x_{k+1}^\delta \in \argmin\, J_k(x)$.

\begin{lemma} \label{lem:ex-min-tikh}
Let assumption (A1) be satisfied. Then each Tikhonov functional $J_k$ in
\eqref{eq:def-Jk} attains a minimizer on $X$.
\end{lemma}
\noindent {\it Proof.} See \cite[Chapter~10]{EngHanNeu96}.
\mbox{} \hfill $\Box$ \medskip

The assertion of Lemma~\ref{lem:ex-min-tikh} still holds true if, instead of
(A1), we assume that the operator $F_{[k]}$ is continuous and weakly closed,
and that $D(F_{[k]})$ is weakly closed \cite{EngHanNeu96}.
In the next lemma we prove an estimate for the residual of the \textsc{iTK}
iteration.

\begin{lemma} \label{lem:resid-estim}
Let $x_k^\delta$ and $\alpha$ be defined by \eqref{eq:itk} and
\eqref{eq:def-alp-tau} respectively. Then
\begin{equation} \label{eq:resid-estim}
 \| F_{[k]}(x_{k+1}^\delta) - y_{[k]}^\delta \|^2 \ \leq \
 \| F_{[k]}(x_{k}^\delta) - y_{[k]}^\delta \|^2 \, , \qquad k < k_*^\delta \, .
\end{equation}
\end{lemma}
\begin{proof}
The inequality in \eqref{eq:resid-estim} is a direct consequence of
$$
\| F_{[k]}(x_{k+1}^\delta) - y_{[k]}^\delta \|^2 \ \leq \ J_k(x_{k+1}^\delta)
\ \leq \ J_k(x_k^\delta)  \ \leq \ \| F_{[k]}(x_{k}^\delta) - y_{[k]}^\delta \|^2
\, , \qquad k < k_*^\delta \, .
$$
\end{proof}

The following lemma is an important auxiliary result, which will be used
to prove a monotony property of the \textsc{iTK} iteration.

\begin{lemma} \label{lem:monot-aux}
Let $x_k^\delta$ and $\alpha$ be defined by \eqref{eq:itk} and
\eqref{eq:def-alp-tau} respectively. Moreover, assume that
(A1) - (A3) hold true.
If $x_{k+1}^\delta \in B_\rho(x_0)$ for some $k \in \N$, then
\begin{equation} \label{eq:monot-aux}
\| x_{k+1}^\delta  - x^* \|^2 - \| x_k^\delta - x^* \|^2  \le
   \frac{2}{\alpha} \norm{ F_{[k]}(x_{k+1}^\delta) - y_{[k]}^\delta }
   \Big[ (\eta-1) \norm{ F_{[k]}(x_{k+1}^\delta) - y_{[k]}^\delta }
         + (1+\eta)\delta_{[k]} \Big] \, .
\end{equation}
\end{lemma}
\begin{proof}
From (\ref{eq:itk}) it follows that
\begin{align*}
\norm{x_{k+1}^\delta - x^*}^2 & - \norm{x_k^\delta - x^*}^2
\\[1ex]
&\quad \le 2\, \langle x_{k+1}^\delta - x^* , \ x_{k+1}^\delta -x_k^\delta \rangle
\\
& \quad = \frac{2}{\alpha} \,
  \langle x_{k+1}^\delta - x^* , \ F'_{[k]}(x_{k+1}^\delta)^*
         (y_{[k]}^\delta - F_{[k]}(x_{k+1}^\delta) ) \rangle
\\
& \quad = \frac{2}{\alpha} \,
  \langle y_{[k]}^\delta - F_{[k]}(x_{k+1}^\delta) , \
          F'_{[k]}(x_{k+1}^\delta) (x_{k+1}^\delta - x^*)
          \pm F_{[k]}(x_{k+1}^\delta) \pm F_{[k]}(x^*) \rangle
\\
 & \quad \leq \frac{2}{\alpha} \,
   \Big( \langle F_{[k]}(x_{k+1}^\delta) - y_{[k]}^\delta, \
         F_{[k]}(x_{k+1}^\delta) - F_{[k]}(x^*) -
         F'_{[k]}(x_{k+1}^\delta) (x_{k+1}^\delta - x^*) \rangle
\\
& \qquad\qquad  + 2
  \langle F_{[k]}(x_{k+1}^\delta) - y_{[k]}^\delta, \
          F_{[k]}(x^*) - F_{[k]}(x_{k+1}^\delta) \pm y_{[k]}^\delta \rangle
          \Big) \, .
\end{align*}
Now, applying the Cauchy-Schwarz inequality and \eqref{eq:a-tcc} with
$x = x^* \in B_{\rho/4}(x_0)$, $\bar x = x_{k+1}^\delta \in B_\rho(x_0)$,
leads to
\begin{align*}
\norm{ x_{k+1}^\delta - x^*}^2 & - \norm{x_k^\delta - x^*}^2
  \le \frac{2}{\alpha} \,
  \norm{ F_{[k]}(x_{k+1}^\delta) - y_{[k]}^\delta }
  \Big( \eta \norm{ F_{[k]}(x_{k+1}^\delta) - y_{[k]} \pm y_{[k]}^\delta }
\\
& \qquad\qquad\qquad\qquad\qquad\qquad\qquad\qquad\qquad
  - \norm{ F_{[k]}(x_{k+1}^\delta) - y_{[k]}^\delta }
  + \norm{ y_{[k]} - y_{[k]}^\delta } \Big) \, ,
\end{align*}
and \eqref{eq:monot-aux} follows from this inequality together with
\eqref{eq:noisy-i}.
\end{proof}

It is worth noticing that the proof of Lemma~\ref{lem:monot-aux} requires
an assumption on $x_{k+1}^\delta$, namely that $x_{k+1}^\delta \in B_\rho(x_0)$.
In the next lemma we make sure that this assumption is satisfied.

\begin{lemma} \label{lem:monot-aux2}
Let $x_k^\delta$ and $\alpha$ be defined by \eqref{eq:itk} and
\eqref{eq:def-alp-tau} respectively. Moreover, assume that (A1),
(A3) hold true.
If $x_k^\delta \in B_{\rho/4}(x^*)$ for some $k \in \N$, then
$x_{k+1}^\delta \in B_\rho(x_0)$.
\end{lemma}
\begin{proof}
It follows from the definition of $x_{k+1}^\delta$ that
$$
\alpha \norm{x_{k+1}^\delta - x_k^\delta}^2
\le J_k(x_{k+1}^\delta) \le J_k(x^*)
\le \norm{ y_{[k]} - y_{[k]}^\delta }^2 + \alpha (\rho/4)^2 \, .
$$
From this inequality and (\ref{eq:def-alp-tau}) we obtain
$\norm{x_{k+1}^\delta - x_k^\delta} \le \delta_{[k]}(\sqrt{\alpha})^{-1}
+ \rho/4 \le \rho/2$. Therefore, it follows that
$$
\norm{x_{k+1}^\delta - x_0}
\le \norm{x_{k+1}^\delta - x_k^\delta} + \norm{x_k^\delta - x_0}
\le \rho/2 + \rho/2 \, ,
$$
completing the proof.
\end{proof}

Our next goal is to prove a monotony property, known to be satisfied
by other iterative regularization methods, e.g., by the Landweber
\cite{EngHanNeu96}, the steepest descent \cite{Sch96}, the
\textsc{LK} \cite{KowSch02} method, the \textsc{l-LK} method \cite{HLS07},
and the \textsc{l-SDK} method \cite{CHLS08}.

\begin{propo}[Monotonicity] \label{prop:monot}
Under the assumptions of Lemma~\ref{lem:monot-aux}, for all $k < k_*^\delta$
the iterates $x_k^\delta$ remain in $B_{\rho/4}(x^*) \subset B_{\rho}(x_0)$
and satisfy \eqref{eq:monot-aux}. Moreover,
\begin{equation} \label{eq:itk-monot}
\| x_{k+1}^\delta - x^* \|^2 \ \le \ \| x_k^\delta - x^* \|^2 \, ,
\quad  k < k_*^\delta \, .
\end{equation}
\end{propo}
\begin{proof}
From (A3) it follows that $x_0 \in B_{\rho/4}(x^*)$.
Moreover, Lemma~\ref{lem:monot-aux2} guarantees that $x_1 \in B_{\rho}(x^*)$.
Therefore, it follows from Lemma~\ref{lem:monot-aux} that
\eqref{eq:monot-aux} holds for $k=0$. Then we conclude from
\eqref{eq:monot-aux} and \eqref{eq:def-discrep} that
\begin{align*}
\| x_1^\delta - x^* \|^2 - \| x_0^\delta - x^* \|^2 \ \le \
   \frac{2}{\alpha} \norm{ F_0(x_1^\delta) - y_0^\delta }
   \delta_0 \Big[ \tau (\eta-1) + (1+\eta) \Big] \, .
\end{align*}
Thus, it follows from \eqref{eq:def-alp-tau} that \eqref{eq:itk-monot}
holds for $k=0$. In particular we have $x_1 \in B_{\rho/4}(x^*)$. The
proof follows now using an inductive argument.
\end{proof}

In the next two sections we provide a complete convergence
analysis for the \textsc{iTK} iteration (see Theorems~\ref{th:exact}
and~\ref{th:noisy-nl} below).

\section{iTK Method: Convergence for exact data}
\label{sec:conv-exact}

Throughout this section, we assume that (A1) - (A3) hold true and that
$x_k^\delta$, $\alpha$ and $\tau$ are defined by \eqref{eq:itk} and
\eqref{eq:def-alp-tau}.
Our main goal in this section is to prove convergence of the \textsc{iTK}
iteration for $\delta_i = 0$, $i=0,\dots,N-1$. For exact data
$y = (y_0, \dots, y_{N-1})$, the iterates in \eqref{eq:itk} are denoted
by $x_k$ to contrast with $x_k^\delta$
in the noisy data case.

\begin{lemma} \label{lem:min-norm}
There exists an $x_0$-minimal norm solution of \eqref{eq:inv-probl} in
$B_{\rho/4}(x_0)$, i.e., a solution $x^\dag$ of \eqref{eq:inv-probl} such that
$\norm{ x^\dag - x_0 } = \inf \{ \norm{x - x_0} : x \in B_{\rho/4}(x_0)$ and
$F(x) = y \}$.
Moreover, $x^\dag $ is the only solution of \eqref{eq:inv-probl} in
$B_{\rho/4}(x_0) \cap \big( x_0 + \ker(F '( x^\dag ))^\perp \big) $.
\end{lemma}
\begin{proof}
Lemma \ref{lem:min-norm} is a consequence of
\cite[Proposition~2.1]{HanNeuSch95}. For a detailed proof we refer the
reader to \cite{KalNeuSch08}.
\end{proof}

Throughout the rest of this article,  $x^\dag$ denotes the $x_0$-minimal
norm solution of \eqref{eq:inv-probl}. We define $e_k := x^\dag - x_k$.
From Proposition \ref{prop:monot} it follows that $\norm {e_k}$ is monotone
non increasing.

Notice that Proposition \ref{prop:monot} guarantees that \eqref{eq:monot-aux}
holds for all $k \in \N$. Since the data is exact, \eqref{eq:monot-aux}
can be rewritten as
$\| x_{k+1}  - x^* \|^2 - \| x_k - x^* \|^2  \le
2 \alpha^{-1} (\eta-1) \norm{ F_{[k]}(x_{k+1}) - y_{[k]}}^2$.
By summing over all $k$, this leads to
\begin{equation} \label{eq:sum-1}
\summ_{k=0}^\infty \norm{ F_{[k]}(x_{k+1}) - y_{[k]} }^2 \ \leq \
          \frac{\alpha}{2 (1-\eta)} \norm{x_0 - x^\dag}^2 \ < \ \infty \, ,
\end{equation}
Equation \eqref{eq:sum-1} and the monotony of $\norm {e_k}$ are the main
arguments in the following proof of the convergence of the \textsc{iTK}
iteration.

\begin{theorem}[Convergence for exact data] \label{th:exact}
For exact data, the iteration $(x_k)$ converges to a solution of
\eqref{eq:inv-probl}, as $k \to \infty$. Moreover, if
\begin{equation} \label{eq:kern-cond}
\nr( \F'(x^\dag) ) \subseteq \nr( \F(x) )
\quad \text{ for all } x \in B_\rho(x_0) \, ,\ i=0,\dots,N-1 \, ,
\end{equation}
then $x_k \to x^\dag$.
\end{theorem}

\begin{proof}
We have already observed that $\norm{e_k}$ decreases monotonically. Therefore,
$\norm{e_k}$ converges to some $\epsilon \geq 0$.
In the following we show that $e_k$ is in fact a Cauchy sequence. This is done
similarly as in the proof of \cite[Theorem~3.3]{CHLS08}. The crucial difference is
the fact that the term $| \langle e_n - e_k , e_n \rangle |$ is here estimated by
\begin{eqnarray}
| \langle e_n - e_k , e_n \rangle |
 & \leq & \summ_{i=k}^{n-1} \alpha^{-1}
        \| \F_{i_1}(x_{i+1}) - y_{i_1} \| \, \| \F_{i_1}'(x_{i+1})
        (x^\dag - x_{i+1}) \| \nonumber \\[-1.5ex]
 &      & \quad + \summ_{i=k}^{l-1} \alpha^{-1} | \F_{i_1}(x_{i+1}) - y_{i_1} \|
        \, \| \F_{i_1}'(x_{i+1}) (x_{i+1} - x_{i^*+1}) \|  \nonumber \\[-1.5ex]
 &      & \qquad \ \ \ + \summ_{i=k}^{l-1} \alpha^{-1}
        \| \F_{i_1}(x_{i+1}) - y_{i_1} \| \,
        \| \F_{i_1}'(x_{i+1}) (x_{i^*+1} - x_{n}) \| \, . \label{eq:est-f0}
\end{eqnarray}
Then, it follows from \eqref{eq:a-tcc} that
\begin{eqnarray}
\| \F_{i_1}'(x_{i+1}) (x^\dag - x_{i+1}) \| & \!\!\! \le & \!\!\!
   (1 + \eta) \| \F_{i_1}(x_{i+1}) - y_{i_1} \| \label{eq:est-f1} \\
\| \F_{i_1}'(x_{i+1}) (x_{i+1} - x_{i^*+1}) \| & \!\!\! \le & \!\!\!
   (1 + \eta) \big( \| \F_{i_1}(x_{i+1}) - y_{i_1} \| +
              \| y_{i_1} - \F_{i_1}(x_{i^*+1}) \| \big) \, . \label{eq:est-f2}
\end{eqnarray}
Moreover, from the definition of the iterated Tikhonov method and
and \eqref{eq:a-dfb} it follows that
\begin{eqnarray}
\| \F_{i_1}'(x_{i+1}) (x_{i^*+1} - x_{n}) \| & \!\!\! \le & \!\!\!
   \alpha^{-1} M^2 \summ_{j = 0}^{N-1}
   \norm{ F_{j}(x_{n_0N+j+1}) - y_j }
   \ \le \ \alpha^{-1} M^2 \gamma \, , \label{eq:est-f3}
\end{eqnarray}
with $\gamma = \gamma(n_0) := \sum_{j=0}^{N-1} \norm{F_{j}(x_{n_0N+j+1}) - y_j}$.
Substituting \eqref{eq:est-f1}, \eqref{eq:est-f2}, \eqref{eq:est-f3} in
\eqref{eq:est-f0} leads to
\begin{align*}
 & | \langle e_n - e_k , e_n \rangle | \\
 & \le \summ_{i_0=k_0}^{n_0}  \summ_{i_1=0}^{N-1}
   \alpha^{-1} \norm{ \F_{i_1}(x_{i_0N + i_1 + 1}) - y_{i_1} } \,
   \Big( 2(1+\eta) \norm{ \F_{i_1}(x_{i_0N + i_1 + 1}) - y_{i_1} }
   + [(1+\eta) + \frac{M^2}{\alpha} ] \gamma \Big)
\end{align*}
(we used the fact that $\| y_{i_1} - \F_{i_1}(x_{i^*+1}) \| \le \gamma$)
and we finally obtain the estimate
\begin{eqnarray*}
| \langle e_n - e_k , e_n \rangle |
 & \le & c \, \summ_{i_0=k_0}^{n_0}
   \summ_{i_1=0}^{N-1} \norm{ \F_{i_1}(x_{i_0N + i_1 +1}) - y_{i_1} }^2
   = \ c \, \summ_{i=k_0}^{n-1} \norm{ \F_{[i]}(x_{i+1}) - y_{[i]} }^2
\end{eqnarray*}
with $c := (N+2) \alpha^{-1} (1 + \eta) + N M^2 \alpha^{-1}$.

The remaining of the argumentation (including the proof of the second assertion)
follows the lines of the proof of \cite[Theorem~3.3]{CHLS08}.
\end{proof}

\section{iTK Method: Convergence for noisy data}
\label{sec:conv-noisy}

Throughout this section, we assume that (A1) - (A3) hold true and that
$x_k^\delta$, $\alpha$ and $\tau$ are defined by \eqref{eq:itk},
and \eqref{eq:def-alp-tau}.
Our main goal in this section is to prove that $x_{k_*^\delta}^\delta$
converges to a solution of \eqref{eq:inv-probl} as $\delta \to 0$,
where $k_*^\delta$ is defined in \eqref{eq:def-discrep}.
Our first goal is to verify the finiteness of the stopping index
$k_*^\delta$.

\begin{propo} \label{prop:st-f}
Assume $\delta_{\rm min} := \min \{ \delta_0, \dots  \delta_{N-1} \} > 0$.
Then $k_*^\delta$ defined in \eqref{eq:def-discrep} is finite.
\end{propo}
\begin{proof}
Assume by contradiction that for every $l \in N$, there exists no
$i(l) \in \set{0,\dots, N-1}$ such that
$\norm{ F_{i(l)} (x_{lN + i(l)+1}^\delta) - y_{i(l)}^\delta } \le \tau\delta_{i(l)}$.
From Proposition~\ref{prop:monot} it follows that \eqref{eq:monot-aux}
can be applied recursively for $k=1, \dots, lN$, and we obtain
$$
- \norm{x_0 - x^*}^2
  \leq  \summ_{k=1}^{lN-1} \displaystyle \frac{2}{\alpha}
        \norm{ F_{[k]}(x_{k+1}^\delta) - y_{[k]}^\delta } \Big[
        (\eta-1) \norm{ F_{[k]}(x_{k+1}^\delta) - y_{[k]}^\delta }
        + (1+\eta) \delta_{[k]} \Big] \, , \quad l \in \N \, .
$$
Using the fact that $\norm{ F_{[k]}(x_{k+1}^\delta) - y_{[k]}^\delta } >
\tau \delta_{[k]}$, we obtain the estimate
\begin{eqnarray} \label{eq:st-finite-aux}
\norm{x_0 - x^*}^2
& \geq & \summ_{k=1}^{lN-1} \displaystyle \frac{2}{\alpha}
         \norm{ F_{[k]}(x_{k+1}^\delta) - y_{[k]}^\delta } \delta_{[k]}
         \Big[ \tau (1-\eta) - (1+\eta) \Big] \nonumber \\
& \geq & \Big[ \tau (1-\eta) - (1+\eta) \Big]
         \frac{2 \tau \delta_{\rm min}^2}{\alpha} \, (lN-1)
         \, , \quad l \in \N \, .
\end{eqnarray}
Due to \eqref{eq:def-alp-tau}, the right hand side of \eqref{eq:st-finite-aux}
tends to $+\infty$ as $l \to \infty$, which gives a contradiction.
Consequently, the minimum in (\ref{eq:def-discrep}) takes a finite value.
\end{proof}

For the rest of this section we assume, additionally to (A1) -- (A3), that
\medskip

\noindent (A4) \ 
The operators $F_i$ in \eqref{eq:inv-probl} and it's derivatives $F'_i$
are Lipschitz continuous, i.e., there exists a constant $L$ such that
$$
\norm{F_i(x) - F_i(\bar x)} + \norm{F'_i(x) - F'_i(\bar x)}
\ \le \ L \, \norm{x - \bar x} \, , \
{\rm for \ all} \ x , \bar x \in B_\rho(x_0) \, .
$$
Moreover, the constants $\alpha$ in \eqref{eq:def-alp-tau} and $M$ in
\eqref{eq:a-dfb} are such that $(\overline M + M) L < \alpha$, where
$\overline M = \overline M(\rho,x_0,y,\Delta) :=
\sup\{ \norm{F_i(x) - y_i^\delta} \, : \ i = 0, \dots, N-1 \, ,
\ x \in B_\rho(x_0) \, , \ \norm{y_i^\delta - y_i} \le \delta_i \, ,
\ |\delta| \le \Delta \}$.

The next result concerns the continuity of $x_{k}^\delta$ at $\delta = 0$
for fixed $k \in \N$.

\begin{lemma} \label{le:cont-nl}
Let $\delta_j = (\delta_{j,0}, \dots, \delta_{j,N-1}) \in (0,\infty)^N$
be given with $\lim_{j\to\infty} \delta_j = 0$. Moreover, let
$y^{\delta_j} = (y_0^{\delta_j}, \dots, y_{N-1}^{\delta_j}) \in Y^N$ be a
corresponding sequence of noisy data satisfying
$$
\norm{ y_i^{\delta_j} - y_i } \le \delta_{j,i} \, , \quad
i = 0, \dots, N-1 \, , \ j \in \N \, .
$$
Then, for each $k \in \N$ we have $\lim_{j\to\infty} x_{k+1}^{\delta_j}
=  x_{k+1}$.
\end{lemma}
\begin{proof}
Notice that the uniqueness of global minimizers of $J_k$ in
\eqref{eq:def-Jk} hold true. Indeed, let $\delta \in (0,\infty)^N$
and $y^\delta \in Y^N$ be given as in \eqref{eq:noisy-i}. If
$x_1$, $x_2 \in B_\rho(x_0)$ are minimizers of $J_k$, we have
\begin{align*}
\norm{ x_1 - x_2}^2
&  = \ \alpha^{-1} \langle F'_{[k]}(x_2)^* (F_{[k]}(x_2) - y_{[k]}^\delta) -
        F'_{[k]}(x_1)^* (F_{[k]}(x_1) - y_{[k]}^\delta) , \ x_1 - x_2  \rangle \\
&  = \ \alpha^{-1} \Big[ \langle F_{[k]}(x_2) - y_{[k]}^\delta , \
       (F'_{[k]}(x_2) - F'_{[k]}(x_1)) \, (x_1 - x_2)  \rangle \\
& \qquad + \langle (F_{[k]}(x_2) - F_{[k]}(x_1)) , \ F'_{[k]}(x_1) (x_1 - x_2)
         \rangle \Big] \\
& \leq \ (\overline M + M) L \alpha^{-1} \norm{ x_1 - x_2}^2 \, ,
\end{align*}
and from (A4) it follows that $x_1 = x_2$.
An immediate consequence of this uniqueness is the fact that the
iterative steps $x_{k+1}^\delta$ in \eqref{eq:itk} are uniquely
defined (see \eqref{eq:def-Jk}).
\medskip

The proof of Lemma~\ref{le:cont-nl} uses an inductive argument in $k$.
First we consider the case $k=0$. Notice that $x^{\delta_j}_0 = x_0$
for $j\in\N$ and we can estimate
\begin{align} \label{eq:cont-nl}
\norm{ x_1^{\delta_j} - x_1}^2
&  =  \ \alpha^{-1} \langle F'_{0}(x_1)^* (F_{0}(x_1) - y_{0}) -
          F'_{0}(x_1^{\delta_j})^* (F_{0}(x_1^{\delta_j}) - y_{0}^{\delta_j}) ,
          \ x_1^{\delta_j}  - x_1  \rangle \nonumber \\
&  =  \ \alpha^{-1} \Big[ \langle F_{0}(x_1) - y_{0} , \
        (F'_{0}(x_1) - F'_{0}(x_1^{\delta_j})) (x_1^{\delta_j}  - x_1) \rangle
        \nonumber \\
& \qquad + \langle F_{0}(x_1) - F_{0}(x_1^{\delta_j}) , \ F'_{0}(x_1^{\delta_j})
         (x_1^{\delta_j}  - x_1) \rangle + \langle y_{0}^{\delta_j} - y_{0} , \
         F'_{0}(x_1^{\delta_j}) (x_1^{\delta_j}  - x_1) \rangle \Big]
         \nonumber \\
& \le \ (\overline M + M) L \alpha^{-1} \norm{x_1^{\delta_j} - x_1}^2
        + M \alpha^{-1} \delta_{j,0} \norm{x_1^{\delta_j} - x_1} \, .
\end{align}
Therefore, it follows from (A4) that $\lim_{j\to\infty} x_{1}^{\delta_j} = x_{1}$.
Next, let $k > 0$ and assume that for all $k' < k$ we have $\lim_{j\to\infty}
x_{k'+1}^{\delta_j} = x_{k'+1}$. Arguing as in \eqref{eq:cont-nl} we obtain
the estimate
$$
\norm{ x_{k+1}^{\delta_j} - x_{k+1}}^2 \ \le \
(\overline M + M) L \alpha^{-1} \norm{x_{k+1}^{\delta_j} - x_{k+1}}^2
 + \Big( M \alpha^{-1} \delta_{j,0} + \norm{x_k^{\delta_j} - x_k} \Big)
 \norm{x_{k+1}^{\delta_j} - x_{k+1}} \, .
$$
From (A4) it follows that
\begin{equation} \label{eq:cont-nl-aux}
[\alpha - (\overline M + M) L] \alpha^{-1} \norm{x_{k+1}^{\delta_j} - x_{k+1}}
\ \le \ M \alpha^{-1} \delta_{j,0} + \norm{x_k^{\delta_j} - x_k}
\end{equation}
and from the induction hypothesis we conclude that $\lim_{j\to\infty}
x_{k+1}^{\delta_j} = x_{k+1}$.
\end{proof}

\begin{theorem}[Convergence for noisy data] \label{th:noisy-nl}
Let $\delta_j = (\delta_{j,0}, \dots$, $\delta_{j,N-1})$ be a given
sequence in $(0,\infty)^N$ with $\lim_{j\to\infty} \delta_j = 0$, and
let $y^{\delta_j} = (y^{\delta_j}_{0}, \dots, y^{\delta_j}_{N-1}) \in Y^N$ be a
corresponding sequence of noisy data satisfying
$\norm{ y_i^{\delta_j} - y_i } \le \delta_{j,i}$, $i = 0, \dots, N-1$, $j \in \N$.
Denote by $k^j_* := k_*(\delta_j, y^{\delta_j})$ the corresponding stopping
index defined in \eqref{eq:def-discrep} and assume that the sequence
$\{ k^j_* \}_{j\in\N}$ is unbounded. Then $x_{k^j_*}^{\delta_j}$ converges
to a solution of \eqref{eq:inv-probl}, as $j \to \infty$.
Moreover, if \eqref{eq:kern-cond} holds, then $x_{k^j_*}^{\delta^j} \to x^\dag$.
\end{theorem}
\begin{proof}
The proof is analogous to the proof of \cite[Theor.~3.6]{CHLS08} and
will be omitted. In the proof, \cite[Theor.~3.5]{CHLS08} has to be replaced by
Lemma~\ref{le:cont-nl} above.
\end{proof}

\begin{remark}
The assumption on the boundedness of the sequence $\{ k^j_* \}_{j\in\N}$
in Theorem~\ref{th:noisy-nl} is crucial for the proof. This assumption is
natural when dealing with ill-posed problems and noisy data, since in
practical applications one generally has $k^\delta_* \to \infty$ as
$\delta \to 0$. A similar assumption is also needed in \cite{KowSch02}
to prove convergence of the Landweber-Kaczmarz iteration for noisy data.

In Section~\ref{sec:litk} we investigate the coupling of the \textsc{iTK}
iteration with a loping strategy, which allow us to drop the above
assumption on the boundedness of $\{ k^j_* \}_{j\in\N}$ and still prove
a semiconvergence result analog to Theorem~\ref{th:noisy-nl}.
\end{remark}

\section{The loping iterated Tikhonov-Kaczmarz method} \label{sec:litk}

Motivated by the ideas in \cite{HLS07,CHLS08,HLR09,BKL09}, we investigate
in this section a \textit{loping iterated Tikhonov-Kaczmarz method}
(\textsc{l-iTK} method) for solving (\ref{eq:inv-probl}). This iterative
method is defined by
\begin{equation} \label{eq:litk}
x_{k+1}^\delta \ = \ x_{k}^\delta - \alpha^{-1} \omega_k F_{[k]}'(x_{k+1}^\delta)^*
                  ( F_{[k]}(x_{k+1}^\delta) - y_{[k]}^\delta ) \, .
\end{equation}
where
\begin{equation} \label{eq:def-omk}
    \omega_k \ := \
    \begin{cases}
           1  & \norm{F_{[k]}(x_{k+1}^\delta) - y_{[k]}^\delta}
                \geq \tau \delta_{[k]} \\
           0  & \text{otherwise}
         \end{cases} \, .
\end{equation}
The positive constants $\alpha$ and $\tau$ are defined as in
\eqref{eq:def-alp-tau}. The meaning of \eqref{eq:litk}, \eqref{eq:def-omk}
is the following: at each iterative step an element $x_{k+1/2} \in D_{[k]}$
satisfying
$$
x_{k+1/2} \ = \ x_{k}^\delta - \alpha^{-1} F_{[k]}'(x_{k+1/2})^*
( F_{[k]}(x_{k+1/2}) - y_{[k]}^\delta )
$$
is computed. If $\| F_{[k]}(x_{k+1/2}) - y_{[k]}^\delta \| \ge \tau \delta_{[k]}$
we set $x_{k+1}^\delta = x_{k+1/2}$, otherwise $x_{k+1}^\delta = x_{k}^\delta$.

For exact data ($\delta = 0$) the \textsc{l-iTK} reduces to the \textsc{iTK}
iteration investigated in the previous sections. For noisy data however,
the \textsc{l-iTK} method is fundamentally different from the \textsc{iTK}
method: The bang-bang relaxation parameter $\omega_k$ effects that the
iterates defined in (\ref{eq:itk}) become stationary if all components of
the residual vector $\norm{F_i(x_k^\delta) - y_i^{\delta}}$ fall below a
pre-specified threshold. This characteristic renders \eqref{eq:itk} a
regularization method, as we shall see in Subsection~\ref{ssec:litk-conv-an}.

\begin{remark}
As observed in Remark~\ref{rem:non-uniq}, the iteration in \eqref{eq:litk}
corresponds to
$x_{k+1}^\delta \ \in \ \argmin \big\{ \omega_k \|F_{[k]}(x) - y_{[k]}^\delta\|^2
+ \alpha \| x - x_k^\delta \| \big\}$ and is not uniquely
defined. For noisy data, a semi-convergence result is obtained under the
smooth assumption (A4) on the functionals $F_i$, which guarantees that the
\textsc{l-iTK} iteration is uniquely defined.
\end{remark}

The \textsc{l-iTK} iteration should be terminated when, for the first time,
all $x_k^\delta$ are equal within a cycle. That is, we stop the iteration at
\begin{equation} \label{eq:def-discr-litk}
  k_*^\delta \ := \ \min \{ lN \in \N: \, x_{lN}^\delta =
                 x_{lN+1}^\delta = \cdots = x_{lN+N-1}^\delta \} \, ,
\end{equation}
Notice that $k_*^\delta$ is the smallest multiple of $N$ such that
\begin{equation} \label{eq:def-discr-litk2}
x_{k_*^\delta}^\delta = x_{k_*^\delta+1}^\delta = \dots = x_{k_*^\delta+N-1}^\delta \, .
\end{equation}

\subsection{Convergence analysis} \label{ssec:litk-conv-an}

In what follows we assume that (A1) -- (A3) and (A4) hold true and
that $x_k^\delta$, $\omega_k$, $\alpha$ and $\tau$ are defined by
\eqref{eq:litk}, \eqref{eq:def-omk} and \eqref{eq:def-alp-tau}.
We start by listing some straightforward facts about the \textsc{l-iTK}
iteration:
\begin{itemize}
\item Lemma~\ref{lem:resid-estim} holds true.
Lemma~\ref{lem:monot-aux} still holds true, but \eqref{eq:monot-aux}
has to be replaced by \vspace{-0.6cm}
\end{itemize}
\begin{equation} \label{eq:monot-aux-litk}
\| x_{k+1}^\delta  - x^* \|^2 - \| x_k^\delta - x^* \|^2  \le
   \frac{2\omega_k}{\alpha} \norm{ F_{[k]}(x_{k+1}^\delta) - y_{[k]}^\delta }
   \Big[ (\eta-1) \norm{ F_{[k]}(x_{k+1}^\delta) - y_{[k]}^\delta }
         + (1+\eta)\delta_{[k]} \Big] \, .
\end{equation}
\begin{itemize}
\item Lemma~\ref{lem:monot-aux2} and Proposition~\ref{prop:monot} hold true.
\item Theorem~\ref{th:exact} holds true (for exact data, the \textsc{l-iTK}
iteration reduces to \textsc{iTK}).
\end{itemize}
Before proving the main semiconvergence theorem we need two auxiliary
results: the first result guarantees that, for noisy data, the stopping
index $k_*^\delta$ in \eqref{eq:def-discr-litk} is finite (compare with
Proposition~\ref{prop:st-f}); the second result is the analogous of
Lemma~\ref{le:cont-nl} for the \textsc{l-iTK} iteration.

\begin{propo} \label{prop:st-f-litk}
Assume $\delta_{\rm min} := \min \{ \delta_0, \dots  \delta_{N-1} \} > 0$.
Then $k_*^\delta$ in \eqref{eq:def-discr-litk} is finite, and
\begin{equation} \label{eq:litk-monot-res}
\norm{ F_{i}(x_{k_*^\delta}^\delta) - y_i^\delta } \ < \
\kappa \tau \delta_i \, , \qquad i = 0, \dots, N-1 \, .
\end{equation}
where $\kappa := [ (1+\eta) + M^2/\alpha ] / (1-\eta)$.
\end{propo}
\begin{proof}
Assume by contradiction that for every $l \in N$, there exists
$i(l) \in \set{0,\dots, N-1}$ such that $x_{lN + i(l)} \not= x_{lN}$.
From Proposition~\ref{prop:monot} it follows that \eqref{eq:monot-aux-litk}
can be applied recursively for $k=1, \dots, lN$, and we obtain
$$
- \norm{x_0 - x^*}^2
  \leq  \summ_{k=1}^{lN-1} 2 \frac{\omega_k}{\alpha}
        \norm{ F_{[k]}(x_{k+1}^\delta) - y_{[k]}^\delta } \Big[
        (\eta-1) \norm{ F_{[k]}(x_{k+1}^\delta) - y_{[k]}^\delta }
        + (1+\eta) \delta_{[k]} \Big] \, , \quad l \in \N \, ,
$$
Using the fact that either $\omega_k=0$ or
$\norm{ F_{[k]}(x_{k+1}^\delta) - y_{[k]}^\delta } > \tau \delta_{[k]}$,
we obtain the estimate
\begin{equation} \label{eq:st-finite-aux1}
\norm{x_0 - x^*}^2 \ \geq \
  \summ_{k=1}^{lN-1} 2 \frac{\omega_{k}}{\alpha}
  \norm{ F_{[k]}(x_{k+1}^\delta) - y_{[k]}^\delta } \delta_{[k]}
  \Big[ \tau (1-\eta) - (1+\eta) \Big] \, .
\end{equation}
Equation (\ref{eq:st-finite-aux1}) and the fact that
$x_{l'N + i(l')} \not= x_{l'N}$ for all $l' \in \N$, imply
\begin{equation} \label{eq:st-finite-aux2}
\norm{x_0 - x^*}^2 \ \geq \
  \Big[ \tau (1-\eta) - (1+\eta) \Big]
  \, 2 l \, \frac{\delta_{\rm min}}{\alpha} \, (\tau \delta_{\rm min})
  \, , \quad l \in \N \, .
\end{equation}
Due to \eqref{eq:def-alp-tau}, the right hand side of \eqref{eq:st-finite-aux2}
tends to $+\infty$ as $l \to \infty$, which gives a contradiction.
Consequently, the set
$\{ l \in \N: x_{lN + i} = x_{lN}$, $0 \leq i \leq N-1 \}$ is not empty
and the minimum in (\ref{eq:def-discrep}) takes a finite value.
\medskip

It remains to prove \eqref{eq:litk-monot-res}.
For each fixed $i \in \set{0,\dots, N-1}$ we have
\begin{align*}
\norm{ F_{i}(x_{k_*^\delta}^\delta) - y_i^\delta }
& \leq  \norm{ F_{i}(x_{k_*^\delta}^\delta) - F_{i}(x_{k_*^\delta+1/2}^\delta)
        + F'_{i}(x_{k_*^\delta+1/2}^\delta)
        ( x_{k_*^\delta+1/2}^\delta - x_{k_*^\delta}^\delta ) } \\
& \quad + \norm{ F_{i}(x_{k_*^\delta+1/2}^\delta) - y_i^\delta }
        + \norm{ -F'_{i}(x_{k_*^\delta+1/2}^\delta)
        ( x_{k_*^\delta+1/2}^\delta - x_{k_*^\delta}^\delta ) } \\
& \leq  \eta \norm{ F_{i}(x_{k_*^\delta}^\delta) - F_{i}(x_{k_*^\delta+1/2}^\delta)
        \pm y_i^\delta } + \tau \delta_i
        + M \norm{ x_{k_*^\delta+1/2}^\delta - x_{k_*^\delta}^\delta } \\
& \leq  \eta \norm{ F_{i}(x_{k_*^\delta}^\delta) - y_i^\delta }
        + (1+\eta) \tau \delta_i + M \alpha^{-1}
        \norm{ F'_{i}(x_{k_*^\delta+1/2}^\delta)
        ( F_{i}(x_{k_*^\delta+1/2}^\delta) - y_i^\delta ) }
\end{align*}
(in the last inequality we used the fact that $\omega_{k_*^\delta+i} = 0$
and $\norm{ F_{i}(x_{k_*^\delta+1/2}^\delta) - y_i^\delta } \le \tau \delta_i$).%
\footnote{Notice that for distinct $i \in \set{0,\dots, N-1}$ the points 
$x_{k_*^\delta+1/2}^\delta$ may be different, since they are minimizers of
the Tikhonov functionals $J_{k_*^\delta+i}(x) :=
\norm{ F_{i}(x) - y_{i}^\delta }^2 + \alpha \norm{x - x_{k_*^\delta}^\delta}^2$.}
Therefore, we obtain the estimate
\begin{equation} \label{eq:def-kappa}
(1-\eta) \norm{ F_{i}(x_{k_*^\delta}^\delta) - y_i^\delta }
\ \leq \ (1+\eta) \tau \delta_i + M^2 \alpha^{-1}
         \norm{ F_{i}(x_{k_*^\delta+1/2}^\delta) - y_i^\delta }
\end{equation}
and \eqref{eq:litk-monot-res} follows.
\end{proof}

\begin{lemma} \label{le:cont-nl-litk}
Let $\delta_j = (\delta_{j,0}, \dots, \delta_{j,N-1}) \in (0,\infty)^N$
be given with $\lim_{j\to\infty} \delta_j = 0$. Moreover, let
$y^{\delta_j} = (y_0^{\delta_j}, \dots, y_{N-1}^{\delta_j}) \in Y^N$ be a
corresponding sequence of noisy data satisfying
$$
\norm{ y_i^{\delta_j} - y_i } \le \delta_{j,i} \, , \quad
i = 0, \dots, N-1 \, , \ j \in \N \, .
$$
Then, for each fixed $k \in \N$ we have $\lim_{j\to\infty} x_{k+1}^{\delta_j}
= x_{k+1}$.
\end{lemma}
\begin{proof}
Arguing as in the first part of the proof of Lemma~\ref{le:cont-nl},
we conclude that the iterative steps $x_{k+1}^\delta$ in \eqref{eq:litk}
-- \eqref{eq:def-omk} are uniquely defined.
\medskip

The proof of Lemma~\ref{le:cont-nl-litk} uses an inductive argument
in $k$. First we take $k=0$ (notice that $x^{\delta_j}_0 = x_0$
for $j\in\N$). We have to consider two cases: If $\omega_0 = 1$, we argue
as in \eqref{eq:cont-nl} and obtain the estimate
\begin{align} \label{eq:case-w1-litk}
\norm{ x_1^{\delta_j} - x_1}
\ \le \ M \, [\alpha - (\overline M + M) L]^{-1} \, \delta_{j,0} \, .
\end{align}
Otherwise, if $\omega_0 = 0$, we have $x_{1}^{\delta_j} = x_0$ and
$\norm{F_{0}(x_{0+1/2}^{\delta_j}) - y_{0}^{\delta_j}} \le
\tau \delta_{j,0}$. Therefore,
\begin{align*}
\norm{ x_1^{\delta_j} - x_1}^2
&  =  \ \alpha^{-1} \langle F'_{0}(x_1)^* (F_{0}(x_1) - y_{0}
        \pm F_0(x_{0}) \pm y_0^{\delta_j}) , \
        x_1^{\delta_j}  - x_1  \rangle \nonumber \\
& \le \ M \alpha^{-1} \norm{x_1^{\delta_j} - x_1} \, \Big\{
        \norm{F_0(x_1) - F_0(x_0)} + \norm{F_0(x_0) - y_0^{\delta_j}}
        + \norm{y_0^{\delta_j} - y_0} \Big\} \nonumber \\
& \le \ (\overline M +M) \alpha^{-1} \norm{x_1^{\delta_j} - x_1} \, \Big\{
        L \norm{x_1 - x_1^\delta} + \norm{F_0(x_0) - y_0^{\delta_j}}
        + \delta_{j,0} \Big\} \, . \nonumber
\end{align*}
Arguing as in \eqref{eq:def-kappa} we estimate $\norm{F_0(x_0) - y_0^{\delta_j}}
\le \kappa\tau\delta_{j,0}$. Therefore, it follows that
\begin{align} \label{eq:case-w0-litk}
\norm{ x_1^{\delta_j} - x_1} \ \le \
\alpha [ \alpha - (\overline M + M) L]^{-1} \, (\kappa \tau + 1) \delta_{j,0} \, .
\end{align}
Thus, it follows from \eqref{eq:case-w1-litk}, 
\eqref{eq:case-w0-litk} and (A4) that $\lim_{j\to\infty}
x_{1}^{\delta_j} =  x_{1}$.

Now, take $k > 0$ and assume that for all $k' < k$ we have $\lim_{j\to\infty}
x_{k'+1}^{\delta_j} = x_{k'+1}$. Once again two cases must be considered:
$\omega_0 = 1$ and $\omega_0 = 0$. Arguing as in the case $k=0$, we obtain
estimates similar to \eqref{eq:case-w1-litk} and \eqref{eq:case-w0-litk}.
Thus, $\lim_{j\to\infty} x_{k+1}^{\delta_j} = x_{k+1}$ follows using the
induction hypothesis (compare with \eqref{eq:cont-nl-aux} and the corresponding
step in the proof of Lemma~\ref{le:cont-nl}).
\end{proof}

We are now ready to state and prove a semiconvergence result for the
\textsc{l-iTK} iteration.

\begin{theorem} \label{th:conv-an-litk}
Let $\delta_j = (\delta_{j,0}, \dots$, $\delta_{j,N-1})$ be a given
sequence in $(0,\infty)^N$ with $\lim_{j\to\infty} \delta_j = 0$, and
let $y^{\delta_j} = (y^{\delta_j}_{0}, \dots, y^{\delta_j}_{N-1}) \in Y^N$ be a
corresponding sequence of noisy data satisfying
$\norm{ y_i^{\delta_j} - y_i } \le \delta_{j,i}$, $i = 0, \dots, N-1$, $j \in \N$.
Denote by $k^j_* := k_*(\delta_j, y^{\delta_j})$ the corresponding stopping
index defined in \eqref{eq:def-discr-litk}. Then $x_{k^j_*}^{\delta_j}$
converges to a solution $x^*$ of \eqref{eq:inv-probl} as $j \to \infty$.
Moreover, if \eqref{eq:kern-cond} holds, then $x_{k^j_*}^{\delta_j}$
converges to $x^\dag$.
\end{theorem}
\begin{proof}
The proof is analogous to the proof of \cite[Theorem~3.6]{CHLS08} and is divided
in two cases. In the second case (the sequence $k^j_*$ is not bounded) one has
to argue with Lemma~\ref{le:cont-nl-litk}.
%
%
\end{proof}

\section{Applications} \label{sec:applications}

In this section we address parameter identification problems in elliptic equations.
In the focus is the question whether the {\em local tangential cone condition}
\eqref{eq:a-tcc} is satisfied.

Part of the following analysis is based on the verification of a stronger condition,
which implies the local tangential cone condition, namely the {\em (adjoint) range
invariance condition}:%
\footnote{For a proof that the local tangential cone condition follows from the
range invariance condition, see \cite{HanNeuSch95}.}
\begin{quote}
There exists a family  of bounded linear operators $R_x : Y \nach Y$ and a positive constant such that
\begin{equation}\label{eq:ric}
F'(x) = R_x F'(x^\dag) \ \text{ and } \ \|R_x - id \| \ge c \|x- x^\dag\|_X
\, ,\ x \in B_\rho(x^0) \, .
\end{equation}
\end{quote}
It is a well known fact that the range invariance condition implies that $\text{range}(F'(x)) = \text{range}(F'(x^\dag))$, $x \in B_\rho(x^0)$.

The model problem under investigation is an elliptic boundary value problem
\begin{eqnarray}\label{eq:bvp1}
-(au_s)_s + (b u)_s + cu  &= &f\,,\, \text{ in }(0,1) \\\label{eq:bvp2}
-\alpha_0 u_s(0) + \beta_0 u(0) & = & g_0, -\alpha_1 u_s(1) + \beta_1 u(0) \; = \; g_1\,.
\end{eqnarray}
Here $f$ is a given function in $L^2(0,1)$  and  $\alpha_i, \beta_i, g_i$ are real
numbers specified below.
To simplify the discussion we consider here the one-dimensional case only, but we shall
give some hints for two- and three-dimensional cases.

The equation in \eqref{eq:bvp1} may be considered as a simplified model for a steady state convection-diffusion equation. The term $cu$ is a production term where the function $c$ depends on properties of the material. The term $ -(au_s)_s + (b u)_s $ results from an ansatz for the flux
$j:= -au_s + b u\,.$ Here $a,b$ are  functions describing the diffusion and convective part, respectively. For a concrete application see for instance \cite{BaK89}, Chapter I.2.

We want to identify the parameters $a,b,c$ from a measurement $u^\delta \in L_2(0,1)$
of the solution $u \in L_2(0,1)$ of the boundary value problem \eqref{eq:bvp1},
\eqref{eq:bvp2}.
We distinguish between three different inverse problems, namely the so called $a/b/c$--problems:
\begin{quote}
{\em The a-problem}: Find $a$ under the assumptions $b \equiv 0$, $c \equiv 0$. \\[1ex]
{\em The b-problem}: Find $b$ under the assumptions $a \equiv 1$, $c \equiv 1$. \\[1ex]
{\em The c-problem}: Find $c$ under the assumptions $a \equiv 1$, $b \equiv 0$.
\end{quote}
Each problem may be presented by a nonlinear equation of the type
$F(x)  = y$ for an appropriately chosen parameter-to-output mapping $F:D \subset
X \to Y$.

The $a$- and $c$-problem are considered in a huge amount of references whereas the
$b$-problem received less attention. It seems that the tangential cone condition for
this problem has not been investigated up to now; we do that below.
A detailed analysis of regularization methods for the identification in elliptic and
parabolic equations can be found in \cite{Bla-K96}.

\subsection{The c-problem}

Let us start the discussion with the {\em c-problem}, the most simple one.
Here the mapping $F$ is defined as follows:
$$
F : D \ni c \mapsto u(c) \in L_2(0,1) \, , \ \ D \subset X:= Y:= L_2(0,1) \, ,
$$
where $u(c)$ solves the boundary value problem
\begin{eqnarray*}
-u_{ss} + cu & = &f \, , \ \text{ in } (0,1) \\
u(0) \ = \ g_0 \, , \ \ u(1) &=& g_1
\end{eqnarray*}
in the weak sense. The domain of definition is chosen as a ball in $X:= L_2(0,1)$
(see \cite{CoK86}):
$$
D := B_\rho(c^0) \ \text{ where } \ c^0 \in L_2(0,1) \, ,
   \ c^0 \ge 0 \text{ a.e. in } (0,1) \, .
$$
Then the mapping $F$ is Fr\'echet-differentiable in $D$ (see \cite{EngHanNeu96,KalNeuSch08})
and we have
$$
F'(c) h = \Gamma(c)^{-1} (-hu(c)) \, , \ \ F'(c)^* w = -u(c) \Gamma(c)^{-1} w \, ,
\ \ h, w  \in L_2(0,1) \, ,
$$
where $\Gamma(c) : H^2(0,1) \cap H_0^1(0,1) \to L_2(0,1)$ is defined by
$\Gamma(c)u := -u_{ss} + cu$. We assume that $c^0$ is chosen such that
$u(c) \ge \kappa$ a.e. for each $c \in D$, where $\kappa$ is a positive constant.
Then we have
\begin{equation}\label{eq:ctcc}
F'(\tilde c) = R(\tilde c,c) F'(c) \, ,\ c, \tilde c \in D \, , 
\end{equation}
with
$$
R(\tilde c,c)^* w = \Gamma(\tilde c)[u(\tilde c) u(c)^{-1} A(\tilde c)^{-1}w ] \, ,
\ w \in L_2(0,1)\, , \quad
\| R(\tilde c,c) - id \| \le \kappa_1 \|\tilde c -  c\| \, , c,\tilde c \in D \, .
$$
Here $\kappa_1$ is a positive constant. As a result, we see that the range invariance
condition is satisfied and the tangential cone condition follows.

\begin{remark}\label{bem:cmore}
The results above hold also in the two- and three-dimensional cases; no further
assumptions are necessary (see, e.g., \cite{Han97,Kal97}). Clearly, the boundary
conditions have now to be considered in the sense of trace operators.
\end{remark}

\subsection{The b-problem}

Here the parameter-to-output mapping $F$ is defined as follows:
$$
F : D \ni b \mapsto u(b) \in L_2(0,1) \, , \ \ D \subset X:= H^1(0,1)\, ,
\ \ Y:= L_2(0,1) \, ,
$$
where $u(b)$ solves the boundary value problem
\begin{eqnarray*}
-u_{ss} + (bu)_s + u & = &f \, , \ \text{ in } (0,1) \\
- u_s(0) + bu(0)\ = \ g_0 \, ,\ \ -u_s(1)+b u(1) &=& g_1
\end{eqnarray*}
in the weak sense. The boundary value problem above is uniquely solvable in $H^1(0,1)$
whenever $\|b\|_X$ is small enough, which can be seen from an application of the
Lax-Milgram-Lemma. Therefore we choose $D$ as a ball $ B_\rho :=
\{ x \in X \, | \ \|x\|_X \le \rho\}$ in $X$ with $\rho$ small enough such that
$u(b)$ is uniquely determined for each $b \in B_\rho$. Additionally, the
assumption that each parameter $b$ belongs to $H^1(0,1)$ ensures that the solution
$u(b)$ is in $H^2(0,1)$.

Let $b \in B_\rho$. Then $F$ is Fr\'echet-differentiable in $b$ and
$F'(b) h = v$, where $v$ solves
\begin{eqnarray} \label{eq:bpres}
- v_ss + (bv)_s + v & = & -(hu)_s \text{ in } (0,1) \, , \\
-v_s + bv \big|_0^1 & = & -hu \big |_0^1
\end{eqnarray}
We want to verify an inequality which leads to the tangential cone condition.
Let $u = u(b)$, $\tilde u =  u(\tilde b)$ with $\tilde b$, $b \in  B_\rho(b^0)$.
Moreover let $v:= F'(b)(\tilde b -b)$. We define the mapping $Q(b): Y \nach H^1(0,1)$
where $\psi := Q(b)w$ solves the boundary value problem
$$
-\psi_{ss} - b\psi_s + \psi = w \text{ in }(0,1) \, , \quad \psi_s(0)= \psi_s(1)=0 \, ,
$$
in a weak sense. Since $b \in H^1(0,1)$ we see that $\psi$ is more regular, namely
$\psi \in H^2(0,1)$.

Let $w \in Y, \|w\|_Y \le 1, $ and let $\psi := Q(b)w$. Then
\begin{eqnarray*}
\ipl \tilde u  - u - F'(b)(\tilde b - b), w \ipr_Y &=& \ipl \tilde u  - u - v, w \ipr_Y \\
&=& \ipl \tilde u - u - v, -\psi_{ss} -b\psi_s + \psi \ipr_Y \\
&=&  \ipl -(\tilde u - u)_{ss} + [b(\tilde u - u)]_s + (\tilde u - u), \psi \ipr_Y \\
&& + \ipl v_{ss} - [bv]_s - v, \psi \ipr_Y + (\tilde b -b) (\tilde u - u) \psi \big |_0^1 \\
&=&  \ipl [(b - \tilde b) \tilde u]_s, \psi \ipr_Y + \ipl [(\tilde b - b)u]_s, \psi \ipr_Y
  + (\tilde b -b) (\tilde u - u) \psi \big |_0^1 \\
&=& \ipl (\tilde b - b)(\tilde u -u) , \psi_s\ipr_Y \, .
\end{eqnarray*}
This implies 
\begin{eqnarray*}
\|F(\tilde b) - F(b) - F'(b)(\tilde b - b)\|_Y
&  =  & \sup_{\|w\|_Y \le 1} |\ipl \tilde u  - u - F'(b)(\tilde u - u), w \ipr_Y| \\
& \le & \sup_{\|w\|_Y \le 1} |\ipl (\tilde b - b)(\tilde u -u) , (Q(b)w)_s\ipr_Y| \\
& \le & \| (\tilde b - b)(\tilde u -u) \|_{L^2(0,1)}
        \sup_{\|w\|_Y \le 1} \| (Q(b)w)_s \|_{L^2(0,1)} \\
& \le & \| \tilde b - b \|_{L^\infty(0,1)} \| \tilde u -u \|_{L^2(0,1)}
        \sup_{\|w\|_Y \le 1} \| Q(b)w \|_{H^1(0,1)} \, ,
\end{eqnarray*}
and we derive the estimate
\begin{equation}\label{eq:bcone}
\|F(\tilde b) - F(b) - F'(b)(\tilde b - b)\|_Y  \le \kappa_2 \|\tilde b -b \|_{H^1(0,1)} \|\tilde u - u\|_{L^2(0,1)} \, ,
\end{equation}
where the constant $\kappa_2$ depends on the norm of the mapping $Q(b)\,.$

\begin{remark}\label{bem:bmore2}
The formulation of the $b$-problem above can be easily generalized to the two-dimensional
case.%
\footnote{Due to the Sobolev embedding theorem of $H^s$ in $L^\infty$, in the two-dimensional
case the parameter space $X$ has to be chosen a a subset of $H^{1+\varepsilon}$, for some
$\varepsilon > 0$.}
The convection term  in this case is $ \partial_1 (bu) +  \partial_2 (bu)$ and again
a scalar function $b$ has to be identified. The situation is different when one models the
first order term in the equation by $b_1 \partial_1 u + b_2 \partial_2 u$ \cite{Isa06}.
Then one has to identify two parameters and the analysis is much more
delicate. It seems that the identification problems has not been considered in the
framework chosen above; see \cite{ChY02} for the investigation of identifiably
for this inverse problem.
\end{remark}

%
%
\subsection{The a-problem}

Here the parameter-to-solution mapping  $F$ is defined by
$$F: D \ni a \mapsto u(a) \in L_2(0,1)\, ,\ D \subset X:= Y:= L_2(0,1) \, , $$
where $u(a) $ solves the boundary value problem
\begin{eqnarray*}
-(au_{s})_s  & = &f \, , \ \text{ in }(0,1) \\
u(0) \ = \ g_0 \, , \ \ u(1) &=& g_1
\end{eqnarray*}
in the weak sense. The domain of definition is chosen as
$$
D := \{ a \in H^1(0,1) \, | \ a(s) \ge \underline{a} \text{ a.e.} \} \, ,
$$
where $\underline{a}$ is a positive constant. One can prove \cite{KalNeuSch08} that
$F$ is Fr\'echet differentiable in $D$ with
\begin{equation} \label{eq:apres}
F'(a) h = A(a)^{-1} ( (-hu(c)_s)_s ) \, , \ \
F'(c)^* w = -J^{-1}[u(a)_s( A(a)^{-1} w)_s]\, , \ \
h, w  \in L_2(0,1) \, ,
\end{equation}
where $A(a) : H^2(0,1) \cap H_0^1(0,1) \to L_2(0,1)$ is defined as $A(a)u := -(au_{s})_s$
and $J: H^2(0,1) \to L_2(0,1)$ is defined by $J\psi := - \psi_{ss} + \psi$
($J$ is the adjoint of the embedding of $H^1(0,1) $ into $L_2(0,1)$). In \cite{KalNeuSch08}
it is shown that the tangential cone condition is satisfied.

\begin{remark} \label{bem:amore}
The results in this section strongly benefit from the fact that the model is
one-dimensional.
One can see this for instance that, due to the choice of the parameter space,
each admissible parameter is a continuous function. In the two- or three-dimensional
case additional assumptions are necessary in order to obtain the same results
(see, e.g., \cite{Han97}).
\end{remark}

\begin{remark}\label{bem:bmore1}
It seems that the range invariance condition cannot be proved (even under stronger
regularity assumptions) for the $a$- and the $b$-problem, respectively; for the
$a$-problem see \cite{HanNeuSch95}. Notice that the presentation of the
Fr\'echet-derivative in \eqref{eq:apres}, \eqref{eq:bpres} cannot be handled in
the same way as in the case of the $c$-problem.
\end{remark}

\section{Conclusions} \label{sec:conclusion}

In this paper we propose a new iterative method for inverse problems
of the form \eqref{eq:inv-probl}, namely the \textsc{iTK} iteration.
In the  case of noisy data, we also propose a loping version of
\textsc{iTK}, namely, the \textsc{l-iTK} iteration.

In the particular case of dealing with a single operator equation ($N=1$ in (2)),
\textsc{iTK} and \textsc{l-iTK} are the same iteration and reduce to the classical
iterated Tikhonov method.
To the best of our knowledge this method has so far been investigated only for
linear problems \cite{BS87, HG98, Lar75} and the convergence analysis for nonlinear
operator equations was still open.

\subsubsection*{Three good reasons for using the loping iteration}

The first reason is a numerical one: \\
Notice that, \eqref{eq:resid-estim} allow us to conclude $\omega_k = 0$
without having to compute $x_{k+1/2}$ at all. Therefore, after a large
number of iterations, $\omega_k$ will vanish for some $k$ within each
iteration cycle and the computational expensive evaluation of $x_{k+1/2}$
(solution of a nonlinear equation) might be loped, making the \textsc{l-iTK}
method in \eqref{eq:litk} a fast alternative to the \textsc{iTK} method as
well as to classical Kaczmarz type methods \cite{KowSch02, Byr09}.
\medskip

\noindent The second reason is of analytical nature: \\
An alternative to relax the assumption on the boundedness of the sequence
$\{ k^j_* \}_{j\in\N}$ in Theorem~\ref{th:noisy-nl} and still prove a
semiconvergence result, is the introduction of the loping strategy above.
This is done in Theorem~\ref{th:conv-an-litk}.
\medskip

\noindent The third reason is of heuristic nature: \\
The rules for choosing the stooping index $k_*^\delta$ in
\eqref{eq:def-discrep} and in \eqref{eq:def-discr-litk} are quite different.
According to \eqref{eq:def-discrep} the \textsc{iTK} iteration should be
stopped when for the first time one of the equations of system
\ref{eq:inv-probl} is satisfied within a specified threshold. Therefore,
at the iteration step $x_{k_*^\delta}^\delta$, we cannot control all the
residuals $\norm{ F_{i}(x_{k}^\delta) - y_{i}^\delta }$ within the cycle. \\
According to \eqref{eq:def-discr-litk} however, the \textsc{l-iTK} iteration
only stops when all the residuals $\norm{ F_{i}(x_{k}^\delta) - y_{i}^\delta }$,
$i = 0,\dots,N-1$ drop below a specified threshold. Consequently, although
the \textsc{l-iTK} iteration needs more steps to reach discrepancy, it
produces an approximate solution $x_{k_*^\delta}^\delta$ which better fits
all the system data.

\section*{Acknowledgments}

We would like to thank Prof.\,M.Burger (M\"unster) for useful and
stimulating discussions.
A.DC. acknowledges support from CNPq grant 474593/2007--0.
The work of A.L. is supported by the Brazilian National Research
Council CNPq, grant 303098/2009--0, and by the Alexander von Humbolt
Foundation AvH.

\bibliographystyle{amsplain}
\bibliography{kaczIT}

\end{document}